\documentclass[12pt,twoside]{amsart}
\usepackage{amsmath, amsthm, amscd, amsfonts, amssymb, graphicx}
\usepackage[bookmarksnumbered, plainpages]{hyperref}

\newtheorem{thm}{Theorem}[section]
\newtheorem{cor}[thm]{Corollary}

\numberwithin{equation}{section}

\begin{document}

{\it To appear in Revista de la Uni$\acute{o}$n Matem$\acute{a}$tica Argentina}

\title{\bf Geometric inequalities for Einstein totally real
submanifolds in a complex space form}
\author{Pan Zhang, Liang Zhang and Mukut Mani Tripathi}

\thanks{{\scriptsize
\hskip -0.4 true cm \textit{2010 Mathematics Subject Classification:}
53C40; 53C42.
}}

\maketitle

\begin{abstract}
Two geometric inequalities are established for Einstein totally real
submanifolds in a complex space form. As immediate applications of
these inequalities, some non-existence results are obtained.
\end{abstract}

\vskip 0.2 true cm

%------------------------------------------------------------------------------------%

\pagestyle{myheadings}
\markboth{\rightline {\scriptsize P. Zhang, L. Zhang and M.M. Tripathi}}
         {\leftline{\scriptsize Geometric inequalities for Einstein totally
         real submanifolds in a complex space form}}

\bigskip
\bigskip

%------------------------------------------------------------------------------------%
%------------------------------------------------------------------------------------%

\section{ Introduction}

According to Chen's cornerstone work \cite{BYC2}, the following
problem is fundamental: {\it to establish simple relationships
between the main intrinsic invariants and the main extrinsic
invariants of Riemannian submanifolds}. The basic relationships
discovered until now are inequalities and the study of this topic has
attracted a lot of attention during the last two decades. Roughly
speaking, there are three main aspects of the study of this topic,
one looking at the new Riemannian invariants introduced by Chen
\cite{BYC3,BYC4,BYC5,BYC7,GKKT,AM,CA,CM,MM1,MM2,PLW}, the other looking at the
DDVV inequalities \cite{DDVV,JZ,LMW,LZ,IM}, and the last looking at
the Casorati curvatures \cite{DS1,JW1,JW2,VBGE,PL}. In this paper, we
are interested in obtaining characterizations of the relationships by
Chen's invariants.

Let $M$ be a Riemannian $n$-manifold and $p$ a point in $M$. Suppose
that $K(\pi)$ is the sectional curvature of $M$ with respect to a
plane section $\pi\subset T_pM$. For each unit tangent vector $X$ of
$M$ at $p$, the Ricci curvature $\textrm{Ric}(X)$ is defined by
\[
\textrm{Ric}(X) = \sum_{j=2}^nK(X\wedge e_j),
\]
where $\{e_1,e_2,\cdots,e_n\}$ is an orthonormal basis of $T_pM$ with
$e_1=X$.

In general, an $n$-dimensional manifold $M$ whose Ricci tensor has an
eigenvalue of multiplicity at least $n-1$ is called quasi-Einstein.
For instance, the Robertson--Walker spacetimes are quasi-Einstein
manifolds. Further, we say that $M$ is an Einstein manifold if
$\textrm{Ric}(X)$ is independent of the choice of the unit vector
$X$. Then for any unit tangent vector $X$ of $M$ at $p$, one has
\[
\textrm{Ric}(X)=\frac{2}{n}\,\tau(p),
\]
where $\tau(p)$ is the scalar curvature at $p$ defined by
\[
\tau(p)=\sum_{1\leq i<j\leq n}K(e_i\wedge e_j).
\]

For a given point $p$ in $M$, let $\pi_1,\cdots,\pi_q$ be $q$
mutually orthogonal plane sections in $T_pM$, where $q$ is a positive
integer $\leq\frac{n}{2}$. Following \cite{BYC3}, we define
\[
K_q^{\textmd{inf}}(p)=\inf_{\pi_1\perp\cdots\perp
\pi_q}\frac{K(\pi_1)+\cdots+K(\pi_q)}{q},
\]
where $\pi_1,\cdots,\pi_q$ run over all mutually orthogonal $q$ plane
sections in $T_pM$. For each positive integer $q\leq\frac{n}{2}$,
define the invariant $\delta_q^{\textrm{Ric}}$ on $M$ by
\[
\delta_q^{\textrm{Ric}}=\sup_{X\in
T_p^1M}\textrm{Ric}(X)-\frac{2q}{n}K_q^{\inf}(p),
\]
where $X$ runs over all unit vectors in $T_{p}^{1}M := \{ X \in
T_{p}M : \Vert X \Vert = 1 \}$.

In \cite{BYC3}, Chen established two inequalities in terms of the
Riemannian invariant $\delta_q^{\textrm{Ric}}$ for Einstein
submanifolds in a real space form. As a natural prolongation, in this
paper, we obtain two inequalities for Einstein totally real
submanifolds in a complex space form. Unlike \cite{BYC3}, we do not
need the algebraic lemma from \cite{BYC4}. Our algebraic techniques
also provide new approaches to establish inequalities obtained in
\cite{BYC3}.

%------------------------------------------------------------------------------------%

\vskip 1 true cm

\section{ Preliminaries}

Let $N^m$ be a complex $m$-dimensional K\"{a}hler manifold, i.e.
$N^m$ is endowed with an almost complex structure $J$ and with a
$J$-Hermitian metric $\widetilde{g}$. By a complex space form
$N^m(4c)$ we mean an $m$-dimensional K\"{a}hler manifold with
constant holomorphic sectional curvature $4c$. A complete simply
connected complex space form $N^m(4c)$ is holomorphically isometric
to the complex Euclidean $m$-plane $\mathbb{C}^m$, the complex
projective $m$-space $\mathbb{C}P^m(4c)$, or a complex hyperbolic
$m$-space $\mathbb{C}H^m(4c)$ according to $c = 0, c > 0$ or $c < 0$,
respectively. Denote by $\widetilde{\nabla}$ its Levi-Civita
connection. The Riemannian curvature tensor field $\widetilde{R}$
with respect to $\widetilde{\nabla}$ has the expression
\begin{align}
\widetilde{R}(\widetilde{X},\widetilde{Y},\widetilde{Z},\widetilde{W})&=c\big(\langle \widetilde{X},\widetilde{Z}\rangle
\langle\widetilde{Y},\widetilde{W}\rangle-\langle\widetilde{X},\widetilde{W}\rangle
\langle\widetilde{Y},\widetilde{Z}\rangle+\langle J\widetilde{X},\widetilde{Z}\rangle \langle J\widetilde{Y},\widetilde{W}\rangle\notag\\
&\quad-\langle J\widetilde{X},\widetilde{W}\rangle \langle J\widetilde{Y},\widetilde{Z}\rangle+2\langle \widetilde{X},J\widetilde{Y}\rangle \langle \widetilde{Z},J\widetilde{W}\rangle\big),\notag
\end{align}
for any vector fields $\widetilde{X}$, $\widetilde{Y}$,
$\widetilde{Z}$, $\widetilde{W}$ on $N^m(4c)$.

Let $M$ be a totally real submanifold in $N^m(4c)$. According to the
behavior of the tangent spaces under the action of $J$, a submanifold
$M$ in $N^{m}(4c)$ is called totally real if the complex structure
$J$ of $N^{m}(4c)$ carries each tangent space $T_{p}M$ of $M$ into
its corresponding normal space $T^{\perp}_{p}M$ \cite{BYC6}. We
denote the Levi-Civita connection of $M$ by $\nabla$ and by $R$ the
curvature tensor on $M$ with respect to $\nabla$.

The formulas of Gauss and Weingarten are given respectively by
\[
\widetilde{\nabla}_XY=\nabla_XY+h(X,Y),\quad
\widetilde{\nabla}_X\xi=-A_{\xi}X+\nabla^{\perp}_X\xi,
\]
for tangent vector fields $X$ and $Y$ and normal vector field $\xi$,
where $\nabla^{\perp}$ is the normal connection and $A$ is the shape
operator. The second fundamental form $h$ is related to $A_{\xi}$ by
\[
\langle h(X,Y),\xi\rangle=\langle A_{\xi}X,Y\rangle.
\]
The mean curvature vector $\overrightarrow{H}$ of $M$ is defined by
\[
\overrightarrow{H} = \frac{1}{n}\verb"trace" \ h,
\]
and we set $H = \Vert \overrightarrow{H} \Vert$ for convenience.

A submanifold $M$ is called pseudo-umbilical if $\overrightarrow{H}$
is nonzero and the shape operator $A_{\overrightarrow{H}}$ at
$\overrightarrow{H}$ is proportional to the identity map. If
$\overrightarrow{H}=0$, we say $M$ is minimal. Besides, $M$ is called
totally geodesic if $h=0$.

For totally real submanifolds, we have \cite{BYC6}
\[
\nabla^{\perp}_XJY=J\nabla_XY,\quad A_{JX}Y=-Jh(X,Y)=A_{JY}X.
\]
The above formulas immediately imply that $\langle h(X,Y),JZ\rangle$
is totally symmetric. Moreover, the Gauss equation is given by
\cite{BYC6}
\begin{align}
R(X,Y,Z,W)&=c\big(\langle X,Z\rangle \langle Y,W\rangle-\langle X,W\rangle \langle Y,Z\rangle\big)\notag\\
&\quad \quad\quad\quad\quad+\langle h(X,Z),h(Y,W)\rangle-\langle
h(X,W),h(Y,Z)\rangle\notag
\end{align}
for all vector fields $X$, $Y$, $Z$, $W$ on $M$,

Choosing a local frame
$$
e_1, \ldots, e_n, e_{n+1},\ldots,e_{m},
$$
$$
e_{m+1}=J(e_1),\ldots, e_{m+n}=J(e_n),
e_{m+n+1}=J(e_{n+1}),\ldots,e_{2m}=J(e_{m})
$$
in $N^m(4c)$ in such a way that, restricted to $M$, $e_{1}, e_{2},
\ldots, e_{n}$ are tangent to $M$. With respect to the local frame of
$N^m(4c)$ chosen above, we denote the coefficients of the second
fundamental form $h$ by $\{h_{ij}^r\}, 1\leq i<j\leq n; n+1\leq r\leq
2m$.

\vskip 1 true cm

\section{The first inequality}

\begin{thm} \label{thm.1}
For any integer $k\geq2$, let $M$ be a $2k$-dimensional Einstein
totally real submanifold of an $m$-dimensional complex space form
$N^{m}(4c)$ of constant holomorphic sectional curvature $4c$. Then we
have
\begin{equation}\label{eq.3.1}
\delta_k^{{\rm Ric}}\leq2(k-1)(c+H^2).
\end{equation}
The equality case of \eqref{eq.3.1} holds if and only if one of the
following two cases occurs\/{\rm :}

\verb"(i)" $M$ is a minimal and Einstein totally real submanifold,
such that, with respect to suitable orthonormal frames
$\{e_1,\ldots,e_{2k},e_{2k+1},\ldots,e_{2m}\}$, the shape operators
of $M$ take the following form:
  $$A_{r}=\left(
  \begin{array}{ccccc}
   A_{1}^r&\cdots&0\\
   \vdots&\ddots&\vdots\\
   0&\cdots&A_{k}^r\\
  \end{array}
\right), \quad r=2k+1,\ldots,2m,
$$
where $A_i^r,i=1,\ldots,k$, are symmetric $2\times2$ submatrices
satisfying $\verb"trace"(A_1^r)=\cdots=\verb"trace"(A_k^r)=0$.

\verb"(ii)"  $M$ is a pseudo-umbilical and Einstein totally real
submanifold, such that, with respect to suitable orthonormal frames
$\{e_1,\ldots,e_{2k},e_{2k+1},\ldots,e_{2m}\}$, the shape operators
of $M$ take the following form:
  $$A_{r}=\left(
  \begin{array}{ccccc}
   A_{1}^r&\cdots&0\\
   \vdots&\ddots&\vdots\\
   0&\cdots&A_{k}^r\\
  \end{array}
\right), \quad r=2k+2,\ldots,2m,
$$
where $A_i^r$, $i=1,\ldots,k$, are symmetric $2\times2$ submatrices
satisfying $\verb"trace"(A_1^r)=\cdots=\verb"trace"(A_k^r)=0$.
\end{thm}

\begin{proof}
For a given point $p$ in $M$, let $\pi_1,\ldots,\pi_k$ be $k$
mutually orthogonal plane sections at $p$. We choose an orthonormal
basis $\{e_1,\ldots,e_{2k}\}$ of $T_pM$ such that
\[
\pi_1 = \textmd{Span}\{e_1,e_2\},
\ldots,\pi_k=\textmd{Span}\{e_{2k-1},e_{2k}\}.
\]
Since $M$ is a $2k$-dimensional Einstein manifold, we have
$\tau=k\textrm{Ric}(X)$. From the definition of
$\delta_k^{\textrm{Ric}}$ and the equation of Gauss, we have
\begin{align} \label{eq.3.2}
k\delta_k^{\textrm{Ric}}&=\tau-[K(\pi_1)+K(\pi_2)+\cdots+K(\pi_k)]\notag\\
&=k(2k-1)c+\sum_{r}\sum_{1\leq i<j\leq 2k}[h_{ii}^rh_{jj}^r-(h_{ij}^r)^2]-\Big\{c+\sum_{r}[h_{11}^rh_{22}^r-(h_{12}^r)^2]\notag\\
&\quad+\cdots+c+\sum_{r}[h_{2k-1,2k-1}^rh_{2k,2k}^r-(h_{2k-1,2k}^r)^2]\Big\}\\
&\leq2k(k-1)c+\sum_{r}[\sum_{1\leq i<j\leq 2k}h_{ii}^rh_{jj}^r-(h_{11}^rh_{22}^r+\cdots+h_{2k-1,2k-1}^rh_{2k,2k}^r)]\notag\\
&=2k(k-1)c+\frac{1}{2}\sum_{r}\big\{(\sum_{i=1}^{2k}h_{ii}^r)^2-[(h_{11}^r+h_{22}^r)^2+\cdots+(h_{2k-1,2k-1}^r+h_{2k,2k}^r)^2]\big\}\notag
\end{align}
Using the Cauchy inequality, we obtain that
\begin{equation} \label{eq.3.3}
(h_{11}^r+h_{22}^r)^2 + \cdots+ (h_{2k-1,2k-1}^r + h_{2k,2k}^r)^2
\geq \frac{1}{k}(\sum_{i=1}^{2k}h_{ii}^r)^2,
\end{equation}
with the equality case of \eqref{eq.3.3} holds if and only if
\begin{equation} \notag
h_{11}^r+h_{22}^r=\cdots=h_{2k-1,2k-1}^r+h_{2k,2k}^r.
\end{equation}
Plunging \eqref{eq.3.3} into \eqref{eq.3.2}, we have
\begin{align}
k\delta_k^{\textrm{Ric}}& \leq 2k(k-1)c + \frac{1}{2}\sum_{r}
\big\{(\sum_{i=1}^{2k}h_{ii}^r)^2
-\frac{1}{k}(\sum_{i=1}^{2k}h_{ii}^r)^2\big\}\notag\\
&=2k(k-1)c+\frac{k-1}{2k}(\sum_{i=1}^{2k}h_{ii}^r)^2\notag\\
&=2k(k-1)c+\frac{k-1}{2k} 4k^2 H^2\notag\\
&=2k(k-1)(c+H^2),\notag
\end{align}
which implies
\[
\delta_k^{\textrm{Ric}}\leq2(k-1)(c+H^2).
\]

Next, we will discuss the equality case. The equality case of
\eqref{eq.3.1} at a point $p\in M$ holds if and only if we have the
equality in \eqref{eq.3.2} and \eqref{eq.3.3}, i.e. with respect to
suitable orthonormal frames, the shape operators take the following
form:
\[
A_{r}=\left(
  \begin{array}{ccccccc}
    A_1^r & \cdots &0\\
  \vdots &\ddots & \vdots\\
    0 & \cdots & A_k^r \\
  \end{array}
\right), \quad r=2k+1,\ldots,2m,
\]
where $A_i^r,i=1,\cdots,k$, are symmetric $2\times2$ submatrices
satisfying
\[
\verb"trace"(A_1^r)=\cdots=\verb"trace"(A_k^r).
\]
The rest of the discussion is similar to that of the proof of
Theorem~1 in \cite{BYC3}.

\end{proof}

\vskip 1 true cm

\section{The second inequality}

\begin{thm} \label{thm.2}
Let $M$ be an $n$-dimensional Einstein totally real submanifold of an
$m$-dimensional complex space form $N^{m}(4c)$. Then for every
positive integer $q<\frac{n}{2}$, we have
\begin{equation} \label{eq.4.1}
\delta_q^{{\rm Ric}}\leq \left( n-1-\frac{2q}{n}\right) c +
\frac{n(n-q-1)}{n-q}H^2.
\end{equation}
The equality case of \eqref{eq.4.1} holds if and only if $M$ is a
totally geodesic submanifold.
\end{thm}

\begin{proof}
Given a point $p$ in $M$ and a positive integer $q<\frac{n}{2}$ , let
$\pi_1,\ldots,\pi_q$ be $q$ mutually orthogonal plane sections of $M$
at $p$. We choose an orthonormal basis of $T_{p}M$ such that
\[
\pi_1=\textmd{Span}\{e_1,e_2\},\ldots,
\pi_q=\textmd{Span}\{e_{2q-1},e_{2q}\}.
\]
Then from the definition of $\delta_q^{\textrm{Ric}}$ we have
\begin{align}\label{eq.4.2}
n\delta_q^{\textrm{Ric}}(p)&=nRic(X)-2[K(\pi_1)+\cdots+K(\pi_q)]\notag\\
&=2q\textrm{Ric}(X)-2[K(\pi_1)+\cdots+K(\pi_q)]+(n-2q)\textrm{Ric}(X).
\end{align}
For convenience, we set
\[
\textrm{I} = 2q\textrm{Ric}(X)-2[K(\pi_1)+\cdots+K(\pi_q)], \qquad
\textrm{II}=(n-2q)\textrm{Ric}(X).
\]
Now we compute $\textrm{I}$ and $\textrm{II}$ one by one. First, we
rewrite $\textrm{I}$ as
\[
\textrm{I} = \sum_{l=1}^q[\textrm{Ric}(e_{2l-1},e_{2l-1}) +
\textrm{Ric}(e_{2l},e_{2l})-2K(\pi_{l})],
\]
which together with the equation of Gauss gives
\begin{align}\label{eq.4.3}
\textrm{I}&\leq 2q(n-2)c+\sum_{r}\big[(\sum_{j\neq1}h_{11}^rh_{jj}^r+\sum_{j\neq2}h_{22}^rh_{jj}^r+\cdots+\sum_{j\neq2q}h_{2q,2q}^rh_{jj}^r)\notag\\
&\quad -2(h_{11}^rh_{22}^r+h_{33}^rh_{44}^r+\cdots+h_{2q-1,2q-1}^rh_{2q,2q}^r)\big]\notag\\
&=2q(n-2)c+\sum_{r}\big[\sum_{1\leq i\leq 2q, \ 2q+1\leq j\leq n}h_{ii}^rh_{jj}^r+2\sum_{1\leq i<j\leq 2q}h_{ii}^rh_{jj}^r\notag\\
&\quad -2(h_{11}^rh_{22}^r+h_{33}^rh_{44}^r+\cdots+h_{2q-1,2q-1}^rh_{2q,2q}^r)\big]\\
&=2q(n-2)c+\sum_{r}\Big\{\sum_{1\leq i\leq 2q, \ 2q+1\leq j\leq n}h_{ii}^rh_{jj}^r+(h_{11}^r+\cdots+h_{2q,2q}^r)^2\notag\\
&\quad -[(h_{11}^r+h_{22}^r)^2+\cdots+(h_{2q-1,2q-1}^r+h_{2q,2q}^r)^2]\Big\}.\notag
\end{align}
On the other hand, we can rewrite $\textrm{II}$ as
\[
\textrm{II} = \textrm{Ric}(e_{2q+1},e_{2q+1}) +
\textrm{Ric}(e_{2q+2},e_{2q+2}) + \cdots+\textrm{Ric}(e_{n},e_{n}),
\]
which together with the equation of Gauss gives
\begin{align}\label{eq.4.4}
\textrm{II}&=(n-2q)(n-1)c+\sum_r\sum_{j\neq 2q+1}[h_{2q+1,2q+1}^rh_{jj}^r-(h_{2q+1,j}^r)^2]\notag\\
&\quad +\cdots+ \sum_r\sum_{j\neq n}[h_{nn}^rh_{jj}^r-(h_{nj}^r)^2]\notag\\
&\leq (n-2q)(n-1)c+\sum_r(\sum_{j\neq 2q+1}h_{2q+1,2q+1}^rh_{jj}^r+\cdots+\sum_{j\neq n}h_{nn}^rh_{jj}^r)\\
&=(n-2q)(n-1)c+\sum_r(2\sum_{2q+1\leq i<j\leq n}h_{ii}^rh_{jj}^r+\sum_{1\leq i\leq 2q, \ 2q+1\leq j\leq n}h_{ii}^rh_{jj}^r).\notag
\end{align}
Plunging \eqref{eq.4.3} and \eqref{eq.4.4} into \eqref{eq.4.2}, we
obtain that
\begin{align}\label{eq.4.5}
n\delta_q^{\textrm{Ric}}(p)&\leq (n^2-n-2q)c+\sum_r(h_{11}^r+\cdots+h_{2q,2q}^r)^2-\sum_r[(h_{11}^r+h_{22}^r)^2\notag\\
&\quad +\cdots+(h_{2q-1,2q-1}^r+h_{2q,2q}^r)^2]+2\sum_{r}\sum_{2q+1\leq i<j\leq n}h_{ii}^rh_{jj}^r\notag\\
&\quad+2\sum_{r}\sum_{1\leq i\leq 2q, \ 2q+1\leq j\leq n}h_{ii}^rh_{jj}^r\notag\\
&=(n^2-n-2q)c+\sum_r(h_{11}^r+\cdots+h_{2q,2q}^r)^2-\sum_r[(h_{11}^r+h_{22}^r)^2\notag\\
&\quad +\cdots+(h_{2q-1,2q-1}^r+h_{2q,2q}^r)^2]+2\sum_{r}\sum_{2q+1\leq i<j\leq n}h_{ii}^rh_{jj}^r\\
&\quad+[n^2H^2-\sum_{r}(h_{11}^r+\cdots+h_{2q,2q}^r)^2-\sum_{r}(h_{2q+1,2q+1}^r+\cdots+h_{nn}^r)^2]\notag\\
&=(n^2-n-2q)c+n^2H^2-\sum_{r}\big[(h_{11}^r+h_{22}^r)^2+(h_{33}^r+h_{44}^r)^2\notag\\
&\quad+\cdots+(h_{2q-1,2q-1}^r+h_{2q,2q}^r)^2+(h_{2q+1,2q+1}^r)^2+\cdots+(h_{nn}^r)^2\big].\notag
\end{align}
From the Cauchy inequality, we know that
\begin{align} \label{eq.4.6}
&(h_{11}^r+h_{22}^r)^2+\cdots+(h_{2q-1,2q-1}^r+h_{2q,2q}^r)^2+(h_{2q+1,2q+1}^r)^2+\cdots+(h_{nn}^r)^2\notag\\
&\quad\quad\quad\geq \frac{1}{n-q}(h_{11}^r+h_{22}^r+\cdots+h_{nn}^r)^2,
\end{align}
with the equality case of \eqref{eq.4.6} holds if and only if
\[
h_{11}^r+h_{22}^r=\cdots=h_{2q-1,2q-1}^r+h_{2q,2q}^r=h_{2q+1,2q+1}^r=\cdots=h_{nn}^r.
\]
Then we plunge \eqref{eq.4.6} into \eqref{eq.4.5}, namely,
\begin{align}
n\delta_q^{\textrm{Ric}}(p)&\leq (n^2-n-2q)c+n^2H^2-\frac{1}{n-q}\sum_{r}(h_{11}^r+h_{22}^r+\cdots+h_{nn}^r)^2\notag\\
&=(n^2-n-2q)c+\frac{n^2(n-q-1)}{n-q}H^2,\notag
\end{align}
which means
\[
\delta_q^{\textrm{Ric}}\leq (n-1-\frac{2q}{n})c+\frac{n(n-q-1)}{n-q}H^2.
\]
Next, we will discuss the equality case. The equality case of
\eqref{eq.4.1} at a point $p\in M$ holds if and only if we have the
equality in \eqref{eq.4.3}, \eqref{eq.4.4} and \eqref{eq.4.6}, i.e.
with respect to suitable orthonormal frames, the shape operators take
the following form:
\[
A_{r}=\left(
  \begin{array}{ccccccc}
   A_1^r & \cdots & 0 & 0 \\
   \vdots&   \ddots & \vdots & \vdots\\
    0 & \cdots &  A_k^r & 0\\
    0 & \cdots &  0& \mu_rE\\
  \end{array}
\right), \quad r=n+1,\cdots,2m,
\]
where $E$ is the $(n-2q)\times(n-2q)$ identity matrix and $A_i^r$,
$i=1,\ldots,k$, are symmetric $2\times 2$ submatrices satisfying
\[
\verb"trace"(A_1^r)=\cdots=\verb"trace"(A_k^r)=\mu_r.
\]
The rest of the discussion is similar to that of the proof of
Theorem~2 in \cite{BYC3}.
\end{proof}

\vskip 1 true cm

\section{Immediate applications}

From Theorems \ref{thm.1} and \ref{thm.2} we obtain immediately the
following.

\begin{cor}
If a Riemannian $n$-manifold $M$ admits a totally real isometric
immersion into a complex Euclidean space which satisfies
$$\delta_q^{{\rm Ric}}>\frac{n(n-q-1)}{n-q}H^2,$$
for some positive integer $q\leq \frac{n}{2}$ at some point, then $M$
is not an Einstein manifold.
\end{cor}

Theorems \ref{thm.1} and \ref{thm.2} also imply the following.

\begin{cor}
If an Einstein $n$-manifold satisfies
$$\delta_q^{{\rm Ric}}>(n-1-\frac{2q}{n})c,$$
for some positive integer $q\leq \frac{n}{2}$ at some point, then it
admits no totally real minimal isometric immersion into a complex
space form of constant holomorphic sectional curvature $4c$
regardless of codimension.
\end{cor}

Besides, from Theorems \ref{thm.1} and \ref{thm.2}, we can also get
Corollary 3 in \cite{BYC3}.

\vskip 1 true cm

\section{Acknowledgements}

The authors were supported in part by NSF in Anhui
(No. 1608085MA03) and NSF for Higher Education in Anhui (No. KJ2014A257). The authors would like to
thank Professor B.Y. Chen for the discussions held on this topic.

\vskip 1 true cm

%-----------------------------------------------------------------------------
%-----------------------------------------------------------------------------

\bigskip
\bigskip

\noindent {\footnotesize {\it P. Zhang} \\
{School of Mathematical Sciences, University of Science and Technology of China}\\
{Anhui 230026, P.R. China}\\
{Email: panzhang@mail.ustc.edu.cn}

\vskip 0.5 true cm

\noindent {\footnotesize {\it L. Zhang} \\
{School of Mathematics and Computer Science, Anhui Normal University}\\
{Anhui 241000, P.R. China}\\
{Email: zhliang43@163.com}

\vskip 0.5 true cm

\noindent {\footnotesize {\it M.M. Tripathi}\\
Department of Mathematics, Faculty of Science, Banaras Hindu University\\
Varanasi 221005, India\\
Email: mmtripathi66@yahoo.com

\end{document}